\documentclass{aic}

\aicAUTHORdetails{%
  title = {A $4$-choosable Graph that is Not $(8:2)$-choosable}, %% please capitalize all significant words
  author = {Zden\v{e}k Dvo\v{r}{\'a}k and Xiaolan Hu and Jean-S{\'e}bastien Sereni},
  plaintextauthor = {Zdenek Dvorak and Xiaolan Hu and Jean-Sebastien Sereni},
  plaintexttitle = {A 4-choosable Graph that is Not (8:2)-choosable}, %%  title without math or LaTeX
  keywords = {$(am:bm)$-conjecture, $(a:b)$-choosable graph, list coloring.},
}   %%% END \aicAUTHORdetails

\aicEDITORdetails{%
   year={2019},
   number={5},
   received={11 June 2018},   % received date: example: 7 January 2017
   published={30 October 2019},  % published date
   doi={10.19086/aic.10811},      % XXX = number of paper, e.g. aic006 for paper#6
}   %%% END \aicEDITORdetails

\usepackage{microtype,tikz}
\newcommand{\unit}{30pt}
\newcommand{\sst}[2]{\left\{#1\,\middle|\,#2\right\}}
\DeclareMathOperator{\ch}{ch}
\newtheorem{theorem}{Theorem}
\newtheorem{corollary}[theorem]{Corollary}
\newtheorem{lemma}[theorem]{Lemma}
\begin{document}

\begin{frontmatter}[classification=text]
%% EDITOR: this will force the keywords to appear right after the Abstract.
%%   If the abstract is too long and would force the keywords off the
%%   front page, please comment out % [classification=text] above
%%   This way the keywords will be floated on the bottom of the first page
%%   even though the Abstract spills over to the next page.

%%% AUTHOR: Title goes here.  This line is optional.  You must use it
%%   if title has footnote attached or requires nontrivial typesetting,
%%   e.g., inclusion of linebreaks to force nice layout.
\title{A $4$-choosable Graph that is Not $(8:2)$-choosable\titlefootnote{This project falls within the scope of L.I.A.\ STRUCO.}} %% please capitalize all significant words

%%% AUTHOR:
%%% List all authors. If you wish, place grant acknowledgements in \thanks.
%%% In brackets include a short tag for each author.
\author[zd]{Zden\v{e}k Dvo\v{r}{\'a}k\thanks{Supported by the Center of
    Excellence -- Inst.\ for Theor.\ Comp.\ Sci., Prague (project P202/12/G061 of
    Czech Science Foundation).}}
\author[xh]{Xiaolan Hu\thanks{Partially supported by NSFC under grant
    number~11601176.}}
\author[jss]{Jean-S{\'e}bastien Sereni\thanks{Partially supported by P.H.C.
    Barrande~40625WH.}}

%%% AUTHOR: Abstract goes here
\begin{abstract}
    In~1980, Erd\H{o}s, Rubin and Taylor asked whether for all positive
    integers~$a$, $b$, and~$m$, every $(a:b)$-choosable graph is also
    $(am:bm)$-choosable.  We provide a negative answer by exhibiting a
    $4$-choosable graph that is not $(8:2)$-choosable.
\end{abstract}
\end{frontmatter}

%%% CUSTOM extra space added so that the text is not
%%% too close to the keywords
\vspace{1em}
%%% AUTHOR: body of paper starts here

Coloring the vertices of a graph with sets of colors (that is, each vertex is
assigned a fixed-size subset of the colors such that adjacent vertices are
assigned disjoint sets) is a fundamental notion, which in particular captures
fractional colorings.  The fractional chromatic number of a graph~$G$ can
indeed be defined to be the infimum (which actually is a minimum) of the
ratios~$a/b$ such that, if every vertex of~$G$ is replaced by a clique of
order~$b$ and every edge of~$G$ is replaced by a complete bipartite graph
between the relevant cliques, then the chromatic number of the obtained graph
is at most~$a$.

In their seminal work on list coloring, Erd\H{o}s, Rubin and
Taylor~\cite{erdosrubintaylor1979} raised several intriguing questions about
the list version of set coloring. Before stating them, let us review the
relevant definitions.

\paragraph{Set coloring.} A function that assigns a set to each vertex of a
graph is a \emph{set coloring} if the sets assigned to adjacent vertices are
disjoint.  For positive integers~$a$ and~$b\le a$, an \emph{$(a:b)$-coloring}
of a graph~$G$ is a set coloring with range~$\binom{\{1,\dotsc, a\}}{b}$,
\emph{i.e.}, a set coloring that to each vertex of~$G$ assigns a $b$-element subset
of~$\{1,\dotsc, a\}$.  The concept of~$(a:b)$-coloring is a generalization of
the conventional vertex coloring. In fact, an $(a:1)$-coloring is exactly an
ordinary proper $a$-coloring.

A \emph{list assignment} for a graph~$G$ is a function~$L$ that to each
vertex~$v$ of~$G$ assigns a set~$L(v)$ of colors.  A set coloring $\varphi$
of~$G$ is an \emph{$L$-set coloring} if $\varphi(v)\subseteq L(v)$ for
every~$v\in V(G)$.  For a positive integer~$b$, we say that~$\varphi$ is an
\emph{$(L:b)$-coloring} of~$G$ if~$\varphi$ is an $L$-set coloring
and~$|\varphi(v)|=b$ for every~$v\in V(G)$.  If such an $(L:b)$-coloring
exists, then~$G$ is \emph{$(L:b)$-colorable}.  For an integer~$a\ge b$, we say
that~$G$ is \emph{$(a:b)$-choosable} if~$G$ is $(L:b)$-colorable for every
list assignment~$L$ such that~$|L(v)|=a$ for each~$v\in V(G)$.  We abbreviate $(L:1)$-coloring,
$(L:1)$-colorable, and~$(a:1)$-choosable to~$L$-coloring, $L$-colorable,
and~$a$-choosable, respectively.

\paragraph{Questions and results.} It is straightforward to see that if a graph
is $(a:b)$-colorable, it is also $(am:bm)$-colorable for every positive
integer~$m$: we can simply replace every color in an $(a:b)$-coloring by~$m$
new colors.  However, this argument fails in the list coloring setting, leading
Erd\H{o}s, Rubin and Taylor~\cite{erdosrubintaylor1979} to ask whether every
$(a:b)$-choosable graph is also $(am:bm)$-choosable whenever~$m\ge 1$. A positive
answer to this question is sometimes referred to as ``the $(am:bm)$-conjecture''.
Using the characterization of $2$-choosable graphs~\cite{erdosrubintaylor1979},
Tuza and Voigt~\cite{tuvoigt} provided a positive answer when~$a=2$ and~$b=1$.
In the other direction, Gutner and Tarsi~\cite{GuTa09} demonstrated that if~$k$
and~$m$ are positive integers and~$k$ is odd, then every $(2mk : mk)$-choosable
graph is also $2m$-choosable.

Formulated differently, the question is to know whether every $(a:b)$-choosable
graph is also~$(c:d)$-choosable whenever~$c/d=a/b$ and~$c\ge a$. This formulation raises the
same question when~$c/d>a/b$, which was also asked by Erd\H{o}s, Rubin and
Taylor~\cite{erdosrubintaylor1979}. About ten years ago, Gutner and
Tarsi~\cite{GuTa09} answered this last question negatively, by studying the
$k$th choice number of a graph for large values of~$k$. More precisely, the
\emph{$k$th choice number} of a graph~$G$ is~$\ch_{:k}(G)$, the least
integer~$a$ for which~$G$ is $(a:k)$-choosable. Their result reads as follows.

\begin{theorem}[Gutner \& Tarsi, 2009]\label{thm-GT}
Let~$G$ be a graph. For every positive real~$\epsilon$, there exists an
    integer~$k_0$ such that $\ch_{:k}(G)\leq  k(\chi(G) +\epsilon )$ for
    every~$k\geq k_0$.
\end{theorem}
As a direct corollary, one deduces that for all integers~$m\ge3$
and~$\ell>m$, there exists a graph that is~$(a:b)$-choosable and
not~$(\ell:1)$-choosable with~$\frac{a}{b}=m$. (To see this, one can for
example apply Theorem~\ref{thm-GT} with~$\varepsilon=1$ to the disjoint union
of a clique of order~$m-1$ and a complete balanced bipartite graph with choice
number~$\ell+1$.)

Another related result that should be mentioned here  was obtained by Alon,
Tuza and Voigt~\cite{alonlchoos}.  They proved that for every graph~$G$,
\[
    \inf\sst{\frac{a}{b}}{\text{$G$ is $(a:b)$-choosable}}=\inf\sst{\frac{a}{b}}{\text{$G$ is $(a:b)$-colorable}}.
\]
In other words, the fractional choice number of a graph equals its fractional chromatic number.

The purpose of our work is to provide a negative answer to Erd\H{o}s, Rubin and
Taylor's question when~$a=4$ and~$b=1$.
\begin{theorem}\label{thm-main}
There exists a graph~$G$ that is $4$-choosable, but not $(8:2)$-choosable.
\end{theorem}

We build such a graph by incrementally combining pieces with certain
properties.  Each piece is defined, and its relevant properties established, in
the forthcoming lemmas.

\paragraph{Gadgets and lemmas.} A \emph{gadget} is a pair~$(G,L_0)$, where~$G$
is a graph and~$L_0$ is an assignment of lists of even size. Given a gadget~$(G,L_0)$,
a \emph{half-list assignment} for~$G$ is a list assignment~$L$ for~$G$ such that
$|L(v)|=|L_0(v)|/2$ for every~$v\in V(G)$.
Let us start the construction by a key observation on list colorings of~$5$-cycles.
\begin{lemma}\label{lemma-pent}
    Consider the gadget~$(C,L_0)$, presented in Figure~\ref{fig-C}, where $C=v_1v_2v_3v_4v_5$ is a $5$-cycle, $L_0(v_1)=\{1,2,5,6\}$, $L_0(v_2)=\{1,4,5,6\}$,
$L_0(v_3)=L_0(v_4)=\{3,4,5,6\}$ and $L_0(v_5)=\{2,4,5,6\}$.  Then~$C$ is $L$-colorable for every half-list assignment~$L$
such that $|L(v_1)\cap L(v_3)|\le 1$, but~$C$ is not $(L_0:2)$-colorable.
\end{lemma}
\begin{proof}
The first statement is well known, but let us give the easy proof for completeness: since $|L(v_1)\cap L(v_3)|\le 1$, we have $|L(v_1)\cup L(v_3)|\ge 3$,
    and thus~$L(v_1)$ or~$L(v_3)$ contains a color~$c_6$ not belonging to~$L(v_2)$.  By symmetry, we can assume that~$c_6\in L(v_1)$.
    We color~$v_1$ by~$c_6$ and then for $i=5,4,3,2$ in order, we color~$v_i$ by a color~$c_i\in L(v_i)\setminus \{c_{i+1}\}$.
    The resulting coloring is proper---we have $c_2\neq c_6$, since $c_6\not\in L(v_2)$.

Suppose now that~$C$ has an $(L_0:2)$-coloring, and for $c\in\{1,\dotsc,6\}$ let~$V_c$ be the set of vertices of~$C$ on which the color~$c$ is used.
Since two colors are used on each vertex of~$C$, we have $\sum_{c=1}^6 |V_c|=10$.  On the other hand,~$V_c$ is an independent set of a $5$-cycle, and thus $|V_c|\le 2$
for every color~$c$.  Furthermore, color~$1$ only appears in the lists of~$v_1$ and~$v_2$, which are adjacent in~$C$. It follows that~$|V_1|\le1$.
The situation is similar for color~$2$, which appears only in the lists of~$v_1$ and~$v_5$, and also for color~$3$, which only appears
in the lists of~$v_3$ and~$v_4$.
Consequently, $\sum_{c=1}^6 |V_c|\le 3\cdot 2+3\cdot 1=9$, which is a contradiction.
\end{proof}
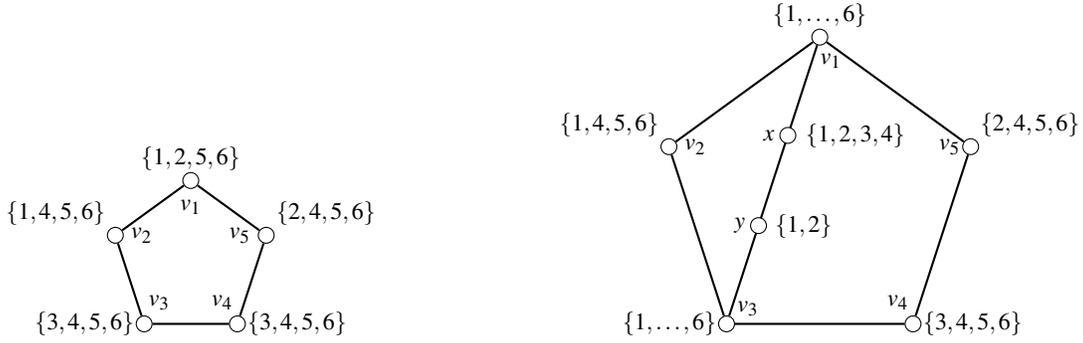
\begin{figure}
\begin{center}
\begin{tikzpicture}[vertex/.style={circle, draw=black, fill=white, inner sep=0.5pt, minimum size=6pt},edge/.style={thick},font=\footnotesize]
 \foreach \i in {1,...,5}
    \draw (18+72*\i:\unit) node[vertex] (a\i) {};
    \draw[edge] (a5)--(a1)--(a2)--(a3)--(a4)--(a5);
\draw[above] (a1) node {$\{1,2,5,6\}$};
\draw[above left] (a2) node {$\{1,4,5,6\}$};
\draw[left] (a3) node {$\{3,4,5,6\}$};
\draw[right] (a4) node {$\{3,4,5,6\}$};
\draw[above right] (a5) node {$\{2,4,5,6\}$};
    \draw (a1)++(-90:3.5mm) node {$v_1$};
\draw (a2)++(0:3.5mm) node {$v_2$};
    \draw (a5)++(0:-3.5mm) node {$v_5$};
\draw (a4)++(-55:-3.5mm) node {$v_4$};
\draw (a3)++(-125:-3.5mm) node {$v_3$};
\end{tikzpicture}
    \hspace{2cm}
\begin{tikzpicture}[vertex/.style={circle, draw=black, fill=white, inner sep=0.5pt, minimum size=6pt},edge/.style={thick},font=\footnotesize]
 \foreach \i in {1,...,5}
    \draw (18+72*\i:2*\unit) node[vertex] (a\i) {};
    \draw[edge] (a5)--(a1)--(a2)--(a3)--(a4)--(a5);
\draw[above] (a1) node {$\{1,\dotsc,6\}$};
\draw[above left] (a2) node {$\{1,4,5,6\}$};
\draw[left] (a3) node {$\{1,\dotsc,6\}$};
\draw[right] (a4) node {$\{3,4,5,6\}$};
\draw[above right] (a5) node {$\{2,4,5,6\}$};
    \draw (a1)++(-65:3.5mm) node {$v_1$};
\draw (a2)++(0:3.5mm) node {$v_2$};
\draw[left] (a5) node {$v_5$};
\draw (a4)++(-55:-3.5mm) node {$v_4$};
\draw (a3)++(215:-3.5mm) node {$v_3$};
\draw[edge] (a1)--(a3);
\path (a1)--(a3) node[pos=1/3,vertex] (x) {} node[pos=2/3,vertex] (y) {};
\draw (x)++(180:2.5mm) node {$x$};
\draw (y)++(180:2.5mm) node {$y$};
\draw (x)++(0:9mm) node {$\{1,2,3,4\}$};
\draw (y)++(0:6mm) node {$\{1,2\}$};
\end{tikzpicture}
    \caption{The gadget~$(C,L_0)$ of Lemma~\ref{lemma-pent} is depicted on the left;
    the gadget~$(G_1,L_1)$ of Corollary~\ref{cor-core} is depicted on the right.}\label{fig-C}
\end{center}
\end{figure}
\begin{corollary}\label{cor-core}
    Consider the gadget~$(G_1,L_1)$, presented in Figure~\ref{fig-C},
    where~$G_1$ consists of a $5$-cycle $C=v_1v_2v_3v_4v_5$ and a
    path~$v_1xyv_3$, with $L_1(v_1)=L_1(v_3)=\{1,\dotsc,6\}$,
    $L_1(v_2)=\{1,4,5,6\}$, $L_1(v_4)=\{3,4,5,6\}$, $L_1(v_5)=\{2,4,5,6\}$,
    $L_1(x)=\{1,2,3,4\}$ and $L_1(y)=\{1,2\}$.  Then $G_1$ is $L$-colorable for
    every half-list assignment~$L$ such that $L(v_1)=L(v_3)$, but $G_1$ is not
    $(L_1:2)$-colorable.
\end{corollary}
\begin{proof}
Let~$L$ be a half-list assignment for~$G_1$. First $L$-color~$y$ and~$x$ by
    colors $c_y\in L(y)$ and~$c_x\in L(x)\setminus\{c_y\}$, respectively.
    Since $c_x\neq c_y$ and~$L(v_1)=L(v_3)$, there exist sets $L'(v_1)\subseteq
    L(v_1)\setminus\{c_x\}$ and~$L'(v_3)\subseteq L(v_3)\setminus\{c_y\}$ of
    size two such that $L'(v_1)\neq L'(v_3)$.  Let~$L'(v_i)=L(v_i)$ for~$i\in
    \{2,4,5\}$. Lemma~\ref{lemma-pent} implies that~$C$ is $L'$-colorable,
    which yields an $L$-coloring of~$G$.

In an $(L_1:2)$-coloring, the vertex~$y$ would have to be assigned~$\{1,2\}$
    and~$x$ would have to be assigned~$\{3,4\}$, and thus the sets of available
    colors for~$v_1$ and for~$v_3$ would have to be~$\{1,2,5,6\}$ and~$\{3,4,5,6\}$,
    respectively.  However, no such $(L_1:2)$ coloring of~$C$ exists according
    to Lemma~\ref{lemma-pent}.
\end{proof}

Next we construct auxiliary gadgets, which will be combined with the gadget from Corollary~\ref{cor-core}
to deal with the case where~$L(v_1)\neq L(v_3)$. 
Let~$G$ be a graph, let~$S$ be a subset of vertices of~$G$ and~$L$ a list assignment for~$G$.
An \emph{$L$-coloring} of~$S$ is a coloring of the subgraph of~$G$ induced by~$S$.
Moreover, if~$S'$ is a subset of vertices of~$G$ that contains~$S$ and~$\varphi'$ is an $L$-coloring of~$S'$,
then~$\varphi'$ \emph{extends} $\varphi$ if~$\varphi'|S=\varphi$.
Let~$(G,L_0)$ be a gadget, let~$v_1$ and~$v_3$ be distinct vertices of~$G$,
and let~$S$ be a set of vertices of~$G$ not containing $v_1$ and~$v_3$.  The gadget is \emph{$(v_1,v_3,S)$-relaxed}
if every half-list assignment~$L$ satisfies at least one of the following conditions.
\begin{itemize}
\item[(i)] There exists an $L$-coloring~$\psi_0$ of~$\{v_1,v_3\}$ such that every $L$-coloring of~$S\cup \{v_1,v_3\}$
extending~$\psi_0$ extends to an $L$-coloring of~$G$.
\item[(ii)] $L(v_1)=L(v_3)$ and there exists an~$L$-coloring~$\psi_0$ of~$S$ such that every $L$-coloring of~$S\cup \{v_1,v_3\}$
extending~$\psi_0$ extends to an $L$-coloring of~$G$.
\end{itemize}

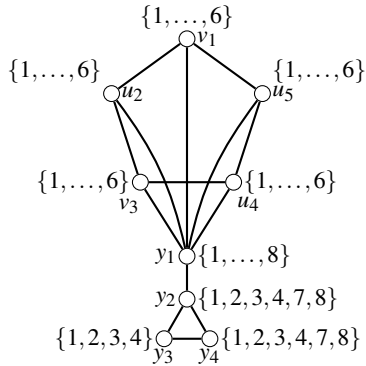
\begin{figure}[!ht]
    \begin{center}
\begin{tikzpicture}[vertex/.style={circle, draw=black, fill=white, inner sep=0.5pt, minimum size=6pt},edge/.style={thick},font=\footnotesize]
 \foreach \i in {1,...,5}
    \draw (18+72*\i:\unit) node[vertex] (a\i) {};
\begin{scope}[yshift={-2*sqrt(\unit*\unit*(1-.588*.588))-\unit}]
\foreach \i in {2,3,4}
  \draw (210+120*\i:\unit/3) node[vertex] (y\i) {};   
 \draw (0,7*\unit/8) node[vertex] (y1) {};
\end{scope}
    \draw[edge] (a5)--(a1)--(a2)--(a3)--(a4)--(a5);
    \draw[edge] (y1) to[bend right=12.5] (a2);
    \draw[edge] (y1) to[bend left=12.5] (a5);
    \draw[edge] (a3)--(y1)--(a1);
    \draw[edge] (y1)--(a4);
\draw[edge] (y1)--(y2);
\draw[edge] (y2)--(y3)--(y4)--(y2);
\draw[above] (a1) node {$\{1,\dotsc,6\}$};
\draw[above left] (a2) node {$\{1,\dotsc,6\}$};
\draw[left] (a3) node {$\{1,\dotsc,6\}$};
\draw[right] (a4) node {$\{1,\dotsc,6\}$};
\draw[above right] (a5) node {$\{1,\dotsc,6\}$};
\draw[right] (y1) node {$\{1,\dotsc,8\}$};
\draw[right] (y2) node {$\{1,2,3,4,7,8\}$};
\draw[right] (y4) node {$\{1,2,3,4,7,8\}$};
\draw[left] (y3) node {$\{1,2,3,4\}$};
\draw[left] (y1) node {$y_1$};
\draw[left] (y2) node {$y_2$};
\draw[below] (y3) node {$y_3$};
\draw[below] (y4) node {$y_4$};
\draw[right] (a1) node {$v_1$};
\draw[right] (a2) node {$u_2$};
\draw[right] (a5) node {$u_5$};
\draw (a4)++(-55:3.5mm) node {$u_4$};
\draw (a3)++(245:3.5mm) node {$v_3$};
\end{tikzpicture}
    \caption{The gadget~$(G_2,L_2)$ of Lemma~\ref{lemma-stagad}}\label{fig-stagad}
\end{center}
\end{figure}

\begin{lemma}\label{lemma-stagad}
    Consider the gadget~$(G_2,L_2)$ presented in Figure~\ref{fig-stagad},
    where~$G_2$ consists of a $5$-cycle $C_2=v_1u_2v_3u_4u_5$, a vertex~$y_1$
    adjacent to all vertices of~$C_2$, a triangle $y_2y_3y_4$, and an
    edge~$y_1y_2$, with $L_2(v)=\{1,\dotsc,6\}$ for every~$v\in V(C_2)$,
    $L_2(y_1)=\{1,\dotsc,8\}$, $L_2(y_2)=L_2(y_4)=\{1,2,3,4,7,8\}$,
    and~$L_2(y_3)=\{1,2,3,4\}$.  The gadget is $(v_1,v_3,\{y_4\})$-relaxed,
    and~$\varphi(y_4)\cap\{7,8\}\neq\varnothing$ for every
    $(L_2:2)$-coloring~$\varphi$ of~$G_2$.
\end{lemma}
\begin{proof}
    Let~$L$ be a half-list assignment for~$G_2$.  If not all vertices of~$C_2$
    have the same list, then choose a color $c\in L(y_1)\setminus L(y_2)$, and
    observe there exists an $L$-coloring of~$G_2[V(C_2)\cup \{y_1\}]$ such
    that~$y_1$ has color~$c$.  Let~$\psi_0$ be the restriction of this coloring
    to~$\{v_1,v_3\}$.  Clearly, every $L$-coloring of~$\{v_1,v_3,y_4\}$
    extending~$\psi_0$ extends to an $L$-coloring of~$G_2$, and thus~(i) holds.

If all the vertices of~$C_2$ have the same list (and hence in particular
    $L(v_1)=L(v_3)$), then let~$c$ be a color in $L(y_1)\setminus L(v_1)$.
    Observe that there exists an $L$-coloring of~$G_2[\{y_1,y_2,y_3,y_4\}]$
    such that~$y_1$ has color~$c$.  Let~$\psi_0$ be the restriction of this
    coloring to~$y_4$.  Again, every $L$-coloring of~$\{v_1,v_3,y_4\}$
    extending~$\psi_0$ extends to an $L$-coloring of~$G_2$, and thus~(ii) holds.

It remains to show that if~$\varphi$ is an $(L_2:2)$-coloring of~$G_2$ then
    $\varphi(y_4)\cap\{7,8\}\neq\varnothing$.  Suppose, on the contrary, that
    $\varphi(y_4)\cap\{7,8\}=\varnothing$. It follows that $\varphi(y_4)\cup
    \varphi(y_3)=\{1,2,3,4\}$, and hence $\varphi(y_2)=\{7,8\}$.  As a result,
    $\varphi(y_1)\subseteq \{1,\dotsc,6\}$ and,  by symmetry, we can assume
    that $\varphi(y_1)=\{5,6\}$.  This implies that
    $\varphi(v)\subseteq\{1,2,3,4\}$ for each~$v\in V(C_2)$. In particular,
    $\varphi|V(C_2)$ is a $(4:2)$-coloring of~$C_2$, which is a
    contradiction since the $5$-cycle $C_2$ has fractional chromatic
    number~$5/2$.
\end{proof}

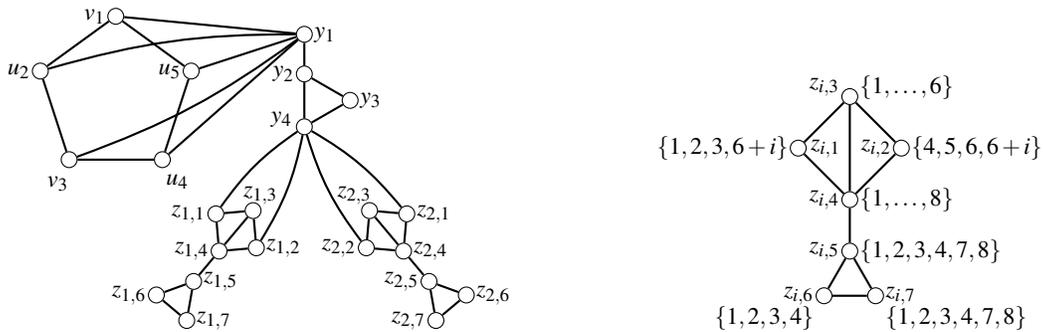
\begin{figure}[!ht]
    \begin{center}
 \begin{tikzpicture}[vertex/.style={circle, draw=black, fill=white, inner sep=0.5pt, minimum size=6pt},edge/.style={thick},font=\footnotesize]
 \foreach \i in {1,...,5}
    \draw (18+72*\i:\unit) node[vertex] (a\i) {};
\draw (18:2.5*\unit) node[vertex] (y1) {};
\begin{scope}[shift={(y1)}]
\begin{scope}[shift={(0,-1*\unit)}]
\path (90:\unit/2) node[vertex] (y2) {};   
\draw (y2)++(0,-2*\unit/3) node[vertex] (y4) {};
\draw (y4)++(30:\unit*2/3) node[vertex] (y3) {};
\end{scope}
\end{scope}
    \draw[edge] (a5)--(a1)--(a2)--(a3)--(a4)--(a5);
    \draw[edge] (y1) to[bend right=8] (a2);
    \draw[edge] (y1) to[bend left=8] (a3);
    \draw[edge] (a5)--(y1)--(a1);
    \draw[edge] (y1)--(a4);
\draw[edge] (y1)--(y2);
\draw[edge] (y2)--(y3)--(y4)--(y2);
\draw[right] (y1) node {$y_1$};
\draw[left] (y2) node {$y_2$};
\draw[right] (y3) node {$y_3$};
\draw (y4)++(-3mm,1mm) node {$y_4$};
\draw[left] (a1) node {$v_1$};
\draw[left] (a2) node {$u_2$};
\draw[left] (a5) node {$u_5$};
\draw (a4)++(-55:3.5mm) node {$u_4$};
\draw (a3)++(245:3.5mm) node {$v_3$};
\begin{scope}[xshift=1.6cm,yshift=-1.8cm,rotate=-40]
\draw (180:\unit/3) node[vertex] (z1) {};
\draw (0:\unit/3) node[vertex] (z2) {};
\draw (90:\unit/3) node[vertex] (z3) {};
\draw (270:\unit/3) node[vertex] (z4) {};
\draw (z4)++(0,-\unit/2) node[vertex] (z5) {};
\draw[edge] (z4)--(z1)--(z3)--(z2)--(z4)--(z3);
\draw (z5)++(-120:\unit/2) node[vertex] (z6) {};
\draw (z5)++(-60:\unit/2) node[vertex] (z7) {};
\end{scope}
\begin{scope}[xshift=3.6cm,yshift=-1.8cm,rotate=40]
\draw (180:\unit/3) node[vertex] (zz2) {};
\draw (0:\unit/3) node[vertex] (zz1) {};
\draw (90:\unit/3) node[vertex] (zz3) {};
\draw (270:\unit/3) node[vertex] (zz4) {};
\draw (zz4)++(0,-\unit/2) node[vertex] (zz5) {};
\draw[edge] (zz4)--(zz1)--(zz3)--(zz2)--(zz4)--(zz3);
\draw (zz5)++(-120:\unit/2) node[vertex] (zz7) {};
\draw (zz5)++(-60:\unit/2) node[vertex] (zz6) {};
\end{scope}
\draw[edge] (y4) to[bend right=10] (z1);
\draw[edge] (y4) to[bend left=10] (z2);
\draw[edge] (z4)--(z5)--(z6)--(z7)--(z5);
\draw[edge] (y4) to[bend right=10] (zz2);
\draw[edge] (y4) to[bend left=10] (zz1);
\draw[edge] (zz4)--(zz5)--(zz6)--(zz7)--(zz5);
\draw[left] (z1) node {$z_{1,1}$};
\draw[right] (zz1) node {$z_{2,1}$};
\draw[right] (z2) node {$z_{1,2}$};
\draw[left] (zz2) node {$z_{2,2}$};
\draw[left] (z4) node {$z_{1,4}$};
\draw[right] (zz4) node {$z_{2,4}$};
\draw[right] (z7) node {$z_{1,7}$};
\draw[left] (zz7) node {$z_{2,7}$};
\draw[left] (z6) node {$z_{1,6}$};
\draw[right] (z5) node {$z_{1,5}$};
\draw[right] (zz6) node {$z_{2,6}$};
\draw[left] (zz5) node {$z_{2,5}$};
\draw (z3)++(1mm,2mm) node {$z_{1,3}$};
\draw (zz3)++(-2mm,2mm) node {$z_{2,3}$};
\end{tikzpicture}
\hspace{1.5cm}
 \begin{tikzpicture}[vertex/.style={circle, draw=black, fill=white, inner sep=0.5pt, minimum size=6pt},edge/.style={thick},font=\footnotesize,scale=.65]
\draw (180:\unit) node[vertex] (z1) {};
\draw (0:\unit) node[vertex] (z2) {};
\draw (90:\unit) node[vertex] (z3) {};
\draw (270:\unit) node[vertex] (z4) {};
\draw (z4)++(0,-\unit) node[vertex] (z5) {};
\draw[edge] (z4)--(z1)--(z3)--(z2)--(z4)--(z3);
\draw (z5)++(-120:\unit) node[vertex] (z6) {};
\draw (z5)++(-60:\unit) node[vertex] (z7) {};
\draw[edge] (z4)--(z5)--(z6)--(z7)--(z5);
     \draw[right] (z1)++(.5mm,0mm) node {$z_{i,1}$};
\draw[left] (z1) node {$\{1,2,3,6+i\}$};
\draw[left] (z2) node {$z_{i,2}$};
\draw[right] (z2) node {$\{4,5,6,6+i\}$};
\draw[left] (z4) node {$z_{i,4}$};
\draw[right] (z4) node {$\{1,\dotsc,8\}$};
\draw[below right] (z7) node {$\{1,2,3,4,7,8\}$};
\draw[right] (z7) node {$z_{i,7}$};
\draw[left] (z6) node {$z_{i,6}$};
\draw[below left] (z6) node {$\{1,2,3,4\}$};
\draw[right] (z5) node {$\{1,2,3,4,7,8\}$};
\draw[left] (z5) node {$z_{i,5}$};
\draw[left] (z3)++(0mm,2mm) node {$z_{i,3}$};
\draw[right] (z3)++(0mm,2mm) node {$\{1,\dotsc,6\}$};
\end{tikzpicture}
    \caption{The graph~$G_3$ from Lemma~\ref{lemma-force78} is depicted on the left; the gadget~$(G_3,L_3)$
    is obtained from~$(G_2,L_2)$ by adding to~$G_3$ a disjoint copy of the
    graph depicted on the right (with the corresponding lists) for
    each~$i\in\{1,2\}$, and joining~$y_4$ to each
    of~$z_{1,1}$,~$z_{1,2}$,~$z_{2,1}$ and~$z_{2,2}$.}\label{fig-force78}
\end{center}
\end{figure}

\begin{lemma}\label{lemma-force78}
Consider the gadget~$(G_3,L_3)$, obtained from the gadget~$(G_2,L_2)$ of
    Lemma~\ref{lemma-stagad} as follows (see Figure~\ref{fig-force78} for an
    illustration of~$G_3$).  The graph~$G_3$ consists of~$G_2$ and
    for~$i\in\{1,2\}$, the vertices~$z_{i,1}, \dotsc, z_{i,7}$; the
    edges~$y_4z_{i,1}$ and~$y_4z_{i,2}$; the edge~$z_{i,j}z_{i,k}$ for
    every~$j$ and every~$k$ such that $1\le j<k\le 4$ and~$(j,k)\neq (1,2)$;
    the edges of the triangle~$z_{i,5}z_{i,6}z_{i,7}$ and the
    edge~$z_{i,4}z_{i,5}$.  Let~$L_3(v)=L_2(v)$ for~$v\in V(G_2)$, and
    for~$i\in\{1,2\}$ let~$L_3(z_{i,1})=\{1,2,3,6+i\}$,
    $L_3(z_{i,2})=\{4,5,6,6+i\}$, $L_3(z_{i,3})=\{1,\dotsc,6\}$,
    $L_3(z_{i,4})=\{1,\dotsc,8\}$, $L_3(z_{i,5})=L_3(z_{i,7})=\{1,2,3,4,7,8\}$
    and $L_3(z_{i,6})=\{1,2,3,4\}$.  The gadget~$(G_3,L_3)$ is
$(v_1,v_3,\{z_{1,7},z_{2,7}\})$-relaxed, and~$\varphi(z_{1,7})=\{7,8\}$ or
$\varphi(z_{2,7})=\{7,8\}$ for every $(L_3:2)$-coloring~$\varphi$ of~$G_3$.
\end{lemma}
\begin{proof} Let~$L$ be a half-list assignment for~$G_3$.
    The gadget~$(G_2,L_3|V(G_2))$ is
    $(v_1,v_3,\{y_4\})$-relaxed by Lemma~\ref{lemma-stagad}.
    Suppose first that~(i) holds for the restriction of~$L$ to~$G_2$
    (with~$S=\{y_4\}$), and let~$\psi_0$ be the corresponding $L$-coloring
    of~$\{v_1,v_3\}$.  For~$i\in\{1,2\}$, if $L(z_{i,1})\cap
    L(z_{i,2})\neq\varnothing$, then let~$c_i$ be a color in~$L(z_{i,1})\cap
    L(z_{i,2})$.  Otherwise, $|L(z_{i,1})\cup
    L(z_{i,2})|=4>|L(z_{i,3})|$, and thus we can choose a color $c_i\in
    (L(z_{i,1})\cup L(z_{i,2}))\setminus L(z_{i,3})$.  Let~$c$ be a color in
    $L(y_4)\setminus \{c_1,c_2\}$.  By~(i) for~$G_2$, we know that~$\psi_0$
    extends to an $L$-coloring~$\psi$ of~$G_2$ such that $\psi(y_4)=c$.  If
    $L(z_{i,1})\cap L(z_{i,2})\neq\varnothing$, then color both~$z_{i,1}$
    and~$z_{i,2}$ by~$c_i$, otherwise color one of them by~$c_i$ and the other
    one by an arbitrary color from its list that is different from~$c$.  There
    are at least two colors in~$L(z_{i,4})$ distinct from the colors
    of~$z_{i,1}$ and~$z_{i,2}$, choose such a color~$c'_i$ so that
    $L(z_{i,5})\setminus \{c'_i\}\neq L(z_{i,6})$.  Color~$z_{i,4}$ by~$c'_i$
    and extend the coloring to~$z_{i,3}$, which is possible by the choice
    of~$c_i$.  Observe that any $L$-coloring of~$z_{i,7}$ extends to an
    $L$-coloring of the triangle~$z_{i,5}z_{i,6}z_{i,7}$ where the color
    of~$z_{i,5}$ is not~$c'_i$. We conclude that~$(G_3,L_3)$ with the half-list
    assignment~$L$ satisfies~(i).

    Suppose next that~(ii) holds for the restriction of~$L$ to~$G_2$ (with~$S=\{y_4\}$), and
    let~$\psi'_0$ be the corresponding $L$-coloring of~$y_4$.  For~$i\in\{1,2\}$,
    greedily extend $\psi'_0$ to an $L$-coloring of~$z_{i,1},
    \dotsc, z_{i,7}$ in order, and let~$\psi_0$ be the restriction of the
    resulting coloring to~$\{z_{1,7},z_{2,7}\}$.  Observe that $(G_3,L_3)$ with
    the half-list assignment~$L$ satisfies~(ii).

    Finally, let~$\varphi$ be an~$(L_3:2)$-coloring of~$G_3$.
    Lemma~\ref{lemma-stagad} implies that
    $\varphi(y_4)\cap\{7,8\}\neq\varnothing$. By symmetry, we can assume that
    $7\in\varphi(y_4)$.  It follows that $\varphi(z_{1,1})\subset \{1,2,3\}$
    and~$\varphi(z_{1,2})\subset \{4,5,6\}$, and thus $\varphi(z_{1,1})\cup
    \varphi(z_{1,2})\cup \varphi(z_{1,3})=\{1,\dotsc,6\}$.  Consequently,
    $\varphi(z_{1,4})=\{7,8\}$, and~$\varphi(z_{1,5})$ is a subset
    of~$\{1,2,3,4\}$.  This yields that $\varphi(z_{1,5})\cup
    \varphi(z_{1,6})=\{1,2,3,4\}$, and therefore~$\varphi(z_{1,7})=\{7,8\}$.
\end{proof}

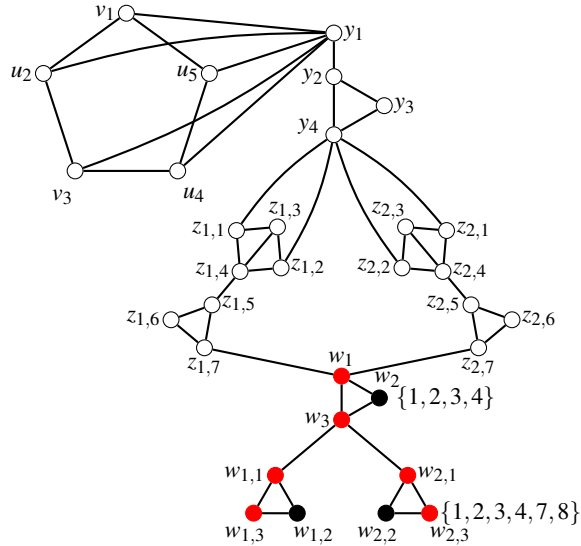
\begin{figure}[!ht]
    \begin{center}
 \begin{tikzpicture}[vertex/.style={circle, draw=black, fill=white, inner sep=0.5pt, minimum size=6pt},edge/.style={thick},scale=1.1,font=\footnotesize]
 \foreach \i in {1,...,5}
    \draw (18+72*\i:\unit) node[vertex] (a\i) {};
\draw (18:2.5*\unit) node[vertex] (y1) {};
\begin{scope}[shift={(y1)}]
\begin{scope}[shift={(0,-1*\unit)}]
\path (90:\unit/2) node[vertex] (y2) {};   
\draw (y2)++(0,-2*\unit/3) node[vertex] (y4) {};
\draw (y4)++(30:\unit*2/3) node[vertex] (y3) {};
\end{scope}
\end{scope}
    \draw[edge] (a5)--(a1)--(a2)--(a3)--(a4)--(a5);
    \draw[edge] (y1) to[bend right=8] (a2);
    \draw[edge] (y1) to[bend left=8] (a3);
    \draw[edge] (a5)--(y1)--(a1);
    \draw[edge] (y1)--(a4);
\draw[edge] (y1)--(y2);
\draw[edge] (y2)--(y3)--(y4)--(y2);
\draw[right] (y1) node {$y_1$};
\draw[left] (y2) node {$y_2$};
\draw[right] (y3) node {$y_3$};
\draw (y4)++(-3mm,1mm) node {$y_4$};
\draw[left] (a1) node {$v_1$};
\draw[left] (a2) node {$u_2$};
\draw[left] (a5) node {$u_5$};
\draw (a4)++(-55:3.5mm) node {$u_4$};
\draw (a3)++(245:3.5mm) node {$v_3$};
\begin{scope}[xshift=1.6cm,yshift=-1.8cm,rotate=-40]
\draw (180:\unit/3) node[vertex] (z1) {};
\draw (0:\unit/3) node[vertex] (z2) {};
\draw (90:\unit/3) node[vertex] (z3) {};
\draw (270:\unit/3) node[vertex] (z4) {};
\draw (z4)++(0,-\unit/2) node[vertex] (z5) {};
\draw[edge] (z4)--(z1)--(z3)--(z2)--(z4)--(z3);
\draw (z5)++(-120:\unit/2) node[vertex] (z6) {};
\draw (z5)++(-60:\unit/2) node[vertex] (z7) {};
\end{scope}
\begin{scope}[xshift=3.6cm,yshift=-1.8cm,rotate=40]
\draw (180:\unit/3) node[vertex] (zz2) {};
\draw (0:\unit/3) node[vertex] (zz1) {};
\draw (90:\unit/3) node[vertex] (zz3) {};
\draw (270:\unit/3) node[vertex] (zz4) {};
\draw (zz4)++(0,-\unit/2) node[vertex] (zz5) {};
\draw[edge] (zz4)--(zz1)--(zz3)--(zz2)--(zz4)--(zz3);
\draw (zz5)++(-120:\unit/2) node[vertex] (zz7) {};
\draw (zz5)++(-60:\unit/2) node[vertex] (zz6) {};
\end{scope}
\draw[edge] (y4) to[bend right=10] (z1);
\draw[edge] (y4) to[bend left=10] (z2);
\draw[edge] (z4)--(z5)--(z6)--(z7)--(z5);
\draw[edge] (y4) to[bend right=10] (zz2);
\draw[edge] (y4) to[bend left=10] (zz1);
\draw[edge] (zz4)--(zz5)--(zz6)--(zz7)--(zz5);
\draw[left] (z1) node {$z_{1,1}$};
\draw[right] (zz1) node {$z_{2,1}$};
\draw[right] (z2) node {$z_{1,2}$};
\draw[left] (zz2) node {$z_{2,2}$};
\draw[left] (z4) node {$z_{1,4}$};
\draw[right] (zz4) node {$z_{2,4}$};
\draw[below] (z7) node {$z_{1,7}$};
\draw[below] (zz7) node {$z_{2,7}$};
\draw[right] (z5) node {$z_{1,5}$};
\draw[left] (z6) node {$z_{1,6}$};
\draw[left] (zz5) node {$z_{2,5}$};
\draw[right] (zz6) node {$z_{2,6}$};
\draw (z3)++(1mm,2mm) node {$z_{1,3}$};
\draw (zz3)++(-2mm,2mm) node {$z_{2,3}$};
\path (z7)--(zz7) node[vertex,midway,below=.25*\unit,red] (w1) {};
\draw (w1)++(0,-\unit/2) node[vertex,red] (w3) {};
\draw (w3)++(30:\unit/2) node[vertex,black] (w2) {};
\path (w1)++(-8mm,-12mm) node[vertex,red] (w11) {};
\path (w1)++(8mm,-12mm) node[vertex,red] (w21) {};
\path (w11)++(60:-\unit/2) node[vertex,red] (w13) {};
\path (w13)++(0:\unit/2) node[vertex,black] (w12) {};
\path (w21)++(60:-\unit/2) node[vertex,black] (w22) {};
\path (w22)++(0:\unit/2) node[vertex,red] (w23) {};`
\draw[edge] (z7)--(w1)--(zz7);
\draw[edge] (w2)--(w1)--(w3)--(w2);
\draw[edge] (w12)--(w13)--(w11)--(w12);
\draw[edge] (w11)--(w3)--(w21);
\draw[edge] (w22)--(w23)--(w21)--(w22);
\draw[above] (w1) node {$w_1$};
\draw (w2)++(1mm,2.25mm) node {$w_2$};
\draw[left] (w3) node {$w_3$};
\draw[left] (w11) node {$w_{1,1}$};
\draw[right] (w21) node {$w_{2,1}$};
\draw (w12)++(2mm,-2.75mm) node {$w_{1,2}$};
\draw (w13)++(-1mm,-2.75mm) node {$w_{1,3}$};
\draw (w23)++(2.5mm,-2.75mm) node {$w_{2,3}$};
\draw (w22)++(-1mm,-2.75mm) node {$w_{2,2}$};
\draw[right] (w23) node {$\{1,2,3,4,7,8\}$};
\draw (w2)++(8mm,0mm) node {$\{1,2,3,4\}$};
\end{tikzpicture}
    \end{center}
    \caption{The graph~$G_4$ from Lemma~\ref{lemma-force78exact}: the filled vertices are those added
    to~$G_3$ to form~$G_4$; the red ones have list~$\{1,2,3,4,7,8\}$ while the black ones have list~$\{1,2,3,4\}$.}\label{fig-force78exact}
\end{figure}

\begin{lemma}\label{lemma-force78exact}
    Consider the gadget~$(G_4,L_4)$ obtained from the gadget~$(G_3,L_3)$ of Lemma~\ref{lemma-force78} as follows (see Figure~\ref{fig-force78exact}
    for an illustration of~$G_4$).
The graph~$G_4$ consists of~$G_3$; the three triangles~$w_1w_2w_3$ and~$w_{i,1}w_{i,2}w_{i,3}$ for~$i\in\{1,2\}$;
and the edges~$z_{1,7}w_1$, $z_{2,7}w_1$, $w_3w_{1,1}$ and~$w_3w_{2,1}$.
Let~$L_4(v)=L_3(v)$ for~$v\in V(G_3)$, and for~$i\in\{1,2\}$, let~$L_4(w_1)=L_4(w_3)=L_4(w_{i,1})=L_4(w_{i,3})=\{1,2,3,4,7,8\}$
and~$L_4(w_2)=L_4(w_{i,2})=\{1,2,3,4\}$.
The gadget is $(v_1,v_3,\{w_{1,3},w_{2,3}\})$-relaxed, and~$\varphi(w_{1,3})=\varphi(w_{2,3})=\{7,8\}$ for
every $(L_4:2)$-coloring~$\varphi$ of~$G_4$.
\end{lemma}
\begin{proof}
Let~$L$ be a half-list assignment for~$G_4$.  The gadget~$(G_3,L_4|V(G_3))$ is
    $(v_1,v_3,\{z_{1,7},z_{2,7}\})$-relaxed by Lemma~\ref{lemma-force78}.
    Suppose first that~(i) holds for the restriction of~$L$ to~$G_3$
    (with~$S=\{z_{1,7},z_{2,7}\}$), and let~$\psi_0$ be the corresponding
    $L$-coloring of~$\{v_1,v_3\}$.  Choose a color $c_1\in L(z_{1,7})$ so that
    $L(w_1)\setminus\{c_1\}\neq L(w_2)$.  If $c_1\in L(w_1)$, then choose
    $c_2\in L(z_{2,7})\setminus (L(w_1)\setminus\{c_1\})$, otherwise choose
    $c_2\in L(z_{2,7})$ so that $L(w_1)\setminus\{c_2\}\neq L(w_2)$.  Choose a
    color $c_3\in L(w_3)$ so that $L(w_{i,1})\setminus\{c_3\}\neq L(w_{i,2})$
    for $i\in\{1,2\}$.  By~(i) for $G_3$, there exists an $L$-coloring of~$G_3$
    extending~$\psi_0$ and assigning~$c_i$ to~$z_{i,7}$ for each~$i\in\{1,2\}$.
    Color~$w_3$ by~$c_3$ and observe that the $L$-coloring can be extended to~$w_1$
    and~$w_2$ thanks to the choice of~$c_1$ and~$c_2$. Moreover, the choice of~$c_3$
    ensures that for each~$i\in\{1,2\}$, we can color~$w_{i,3}$ with any
    color in~$L(w_{i,3})$ and further extend the coloring to~$w_{i,1}$
    and~$w_{i,2}$.  We conclude that $(G_4,L_4)$ with the half-list
    assignment~$L$ satisfies~(i).

Suppose next that~(ii) holds for the restriction of~$L$ to~$G_4$
    (with~$S=\{z_{1,7},z_{2,7}\}$), and let~$\psi'_0$ be the corresponding
    $L$-coloring of~$\{z_{1,7},z_{2,7}\}$.  Greedily extend $\psi'_0$ to an
    $L$-coloring of~$w_1$, $w_2$, $w_3$, and~$w_{i,1}$, $w_{i,2}$, $w_{i,3}$
    for $i\in\{1,2\}$ in order, and let~$\psi_0$ be the restriction of the
    resulting coloring to~$\{w_{1,3},w_{2,3}\}$.  Observe that $(G_4,L_4)$ with
    the half-list assignment~$L$ satisfies~(ii).

Finally, let~$\varphi$ be an~$(L_4:2)$-coloring of~$G_4$.  By
    Lemma~\ref{lemma-force78} and by symmetry, we can assume that
    $\varphi(z_{1,7})=\{7,8\}$.  Consequently, $\varphi(w_1)\subset
    \{1,2,3,4\}$, and thus $\varphi(w_1)\cup \varphi(w_2)=\{1,2,3,4\}$, which
    yields that $\varphi(w_3)=\{7,8\}$.  We conclude analogously that
    $\varphi(w_{1,3})=\{7,8\}=\varphi(w_{2,3})$.
\end{proof}

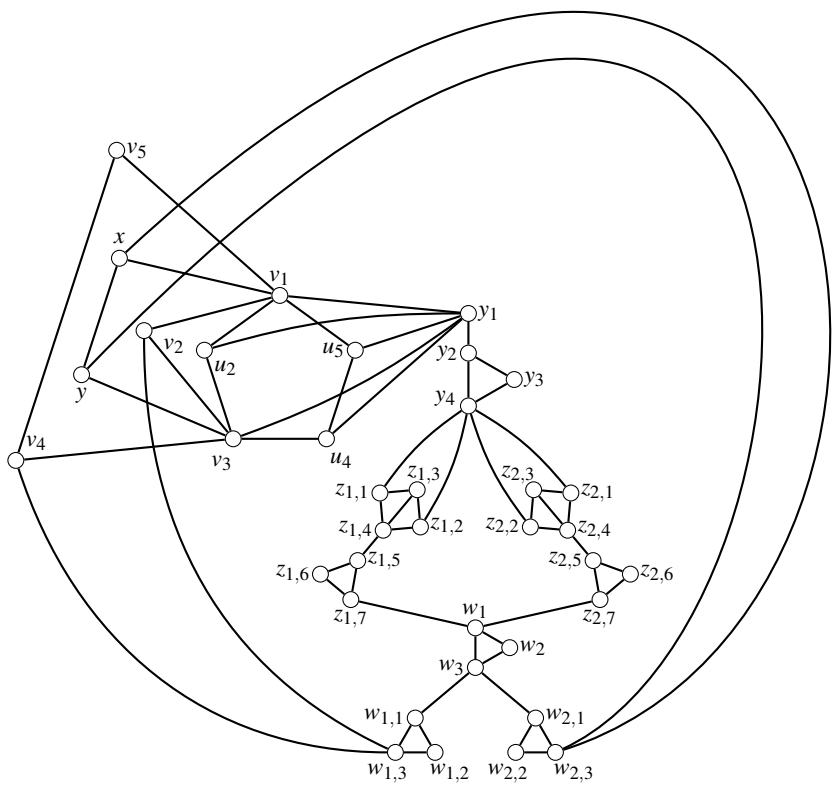
\begin{figure}[!ht]
    \begin{center}
 \begin{tikzpicture}[vertex/.style={circle, draw=black, fill=white, inner sep=0.5pt, minimum size=6pt},edge/.style={thick},font=\footnotesize]
\clip (-3.8,-5.8) rectangle (7.5,5);
\foreach \i in {1,...,5}
    \draw (18+72*\i:\unit) node[vertex] (a\i) {};
\draw (18:2.5*\unit) node[vertex] (y1) {};
\begin{scope}[shift={(y1)}]
\begin{scope}[shift={(0,-1*\unit)}]
\path (90:\unit/2) node[vertex] (y2) {};   
\draw (y2)++(0,-2*\unit/3) node[vertex] (y4) {};
\draw (y4)++(30:\unit*2/3) node[vertex] (y3) {};
\end{scope}
\end{scope}
    \draw[edge] (a5)--(a1)--(a2)--(a3)--(a4)--(a5);
    \draw[edge] (y1) to[bend right=8] (a2);
    \draw[edge] (y1) to[bend left=8] (a3);
    \draw[edge] (a5)--(y1)--(a1);
    \draw[edge] (y1)--(a4);
\draw[edge] (y1)--(y2);
\draw[edge] (y2)--(y3)--(y4)--(y2);
\draw[right] (y1) node {$y_1$};
\draw[left] (y2) node {$y_2$};
\draw[right] (y3) node {$y_3$};
\draw (y4)++(-3mm,1mm) node {$y_4$};
\draw[above] (a1) node {$v_1$};
\draw[below right] (a2) node {$u_2$};
\draw[left] (a5) node {$u_5$};
\draw (a4)++(-55:3.5mm) node {$u_4$};
\draw (a3)++(245:3.5mm) node {$v_3$};
\begin{scope}[xshift=1.6cm,yshift=-1.8cm,rotate=-40]
\draw (180:\unit/3) node[vertex] (z1) {};
\draw (0:\unit/3) node[vertex] (z2) {};
\draw (90:\unit/3) node[vertex] (z3) {};
\draw (270:\unit/3) node[vertex] (z4) {};
\draw (z4)++(0,-\unit/2) node[vertex] (z5) {};
\draw[edge] (z4)--(z1)--(z3)--(z2)--(z4)--(z3);
\draw (z5)++(-120:\unit/2) node[vertex] (z6) {};
\draw (z5)++(-60:\unit/2) node[vertex] (z7) {};
\end{scope}
\begin{scope}[xshift=3.6cm,yshift=-1.8cm,rotate=40]
\draw (180:\unit/3) node[vertex] (zz2) {};
\draw (0:\unit/3) node[vertex] (zz1) {};
\draw (90:\unit/3) node[vertex] (zz3) {};
\draw (270:\unit/3) node[vertex] (zz4) {};
\draw (zz4)++(0,-\unit/2) node[vertex] (zz5) {};
\draw[edge] (zz4)--(zz1)--(zz3)--(zz2)--(zz4)--(zz3);
\draw (zz5)++(-120:\unit/2) node[vertex] (zz7) {};
\draw (zz5)++(-60:\unit/2) node[vertex] (zz6) {};
\end{scope}
\draw[edge] (y4) to[bend right=10] (z1);
\draw[edge] (y4) to[bend left=10] (z2);
\draw[edge] (z4)--(z5)--(z6)--(z7)--(z5);
\draw[edge] (y4) to[bend right=10] (zz2);
\draw[edge] (y4) to[bend left=10] (zz1);
\draw[edge] (zz4)--(zz5)--(zz6)--(zz7)--(zz5);
\draw[left] (z1) node {$z_{1,1}$};
\draw[right] (zz1) node {$z_{2,1}$};
\draw[right] (z2) node {$z_{1,2}$};
\draw[left] (zz2) node {$z_{2,2}$};
\draw[left] (z4) node {$z_{1,4}$};
\draw[right] (zz4) node {$z_{2,4}$};
\draw[below] (z7) node {$z_{1,7}$};
\draw[below] (zz7) node {$z_{2,7}$};
\draw[right] (z5) node {$z_{1,5}$};
\draw[left] (z6) node {$z_{1,6}$};
\draw[left] (zz5) node {$z_{2,5}$};
\draw[right] (zz6) node {$z_{2,6}$};
\draw (z3)++(1mm,2mm) node {$z_{1,3}$};
\draw (zz3)++(-2mm,2mm) node {$z_{2,3}$};
\path (z7)--(zz7) node[vertex,midway,below=.25*\unit] (w1) {};
\draw (w1)++(0,-\unit/2) node[vertex] (w3) {};
\draw (w3)++(30:\unit/2) node[vertex] (w2) {};
\path (w1)++(-8mm,-12mm) node[vertex] (w11) {};
\path (w1)++(8mm,-12mm) node[vertex] (w21) {};
\path (w11)++(60:-\unit/2) node[vertex] (w13) {};
\path (w13)++(0:\unit/2) node[vertex] (w12) {};
\path (w21)++(60:-\unit/2) node[vertex] (w22) {};
\path (w22)++(0:\unit/2) node[vertex] (w23) {};`
\draw[edge] (z7)--(w1)--(zz7);
\draw[edge] (w2)--(w1)--(w3)--(w2);
\draw[edge] (w12)--(w13)--(w11)--(w12);
\draw[edge] (w11)--(w3)--(w21);
\draw[edge] (w22)--(w23)--(w21)--(w22);
\draw[above] (w1) node {$w_1$};
\draw[right] (w2) node {$w_2$};
\draw[left] (w3) node {$w_3$};
\draw[left] (w11) node {$w_{1,1}$};
\draw (w12)++(2mm,-2.75mm) node {$w_{1,2}$};
\draw (w13)++(-1mm,-2.75mm) node {$w_{1,3}$};
\draw (w23)++(2.5mm,-2.75mm) node {$w_{2,3}$};
\draw (w22)++(-1mm,-2.75mm) node {$w_{2,2}$};
\draw[right] (w21) node {$w_{2,1}$};
\draw (126:3.5*\unit) node[vertex] (v5) {};
\draw (144:2.5*\unit) node[vertex] (x) {};
\draw (162:1.8*\unit) node[vertex] (v2) {};
\draw (180:2.5*\unit) node[vertex] (y) {};
\draw (198:3.5*\unit) node[vertex] (v4) {};
\draw[edge] (a1)--(v2)--(a3)--(v4)--(v5)--(a1);
\draw[edge] (a1)--(x)--(y)--(a3);
\draw (v2)++(4mm,-2mm) node {$v_2$};
\draw (x)++(0,2.75mm) node {$x$};
\draw (y)++(0,-2.75mm) node {$y$};
\draw[right] (v5) node {$v_5$};
\draw[above right] (v4) node {$v_4$};
\draw[edge] (w13) to[out=180,in=-72] (v4);
\draw[edge] (w13) to[out=155,in=-90] (v2);
\draw[edge] (w23) ..controls (8,-3) and (8,11).. (y);
\draw[edge] (w23) ..controls (10,-3) and (8,11).. (x);
\end{tikzpicture}
    \caption{The graph~$G_5$ from Lemma~\ref{lemma-main}.}\label{fig-main}
    \end{center}
\end{figure}

We can now combine~$(G_1,L_1)$ with~$(G_4,L_4)$ to obtain a gadget~$(G_5,L_5)$ that is $L$-colorable from each half-list assignment~$L$, but not $(L_5:2)$-colorable.
\begin{lemma}\label{lemma-main}
Consider the gadget~$(G_5,L_5)$ obtained from the gadgets~$(G_1,L_1)$ of
    Corollary~\ref{cor-core} and~$(G_4,L_4)$ of Lemma~\ref{lemma-force78exact}
    as follows (see Figure~\ref{fig-main} for an illustration of~$G_5$).  The
    graph~$G_5$ is obtained from the union of the graphs $G_1$ and~$G_4$
    (intersecting in $\{v_1,v_3\}$) by adding the edges $w_{1,3}v_2$,
    $w_{1,3}v_4$, $w_{2,3}x$ and~$w_{2,3}y$.  Let~$L_5(v)=L_4(v)$ for $v\in
    V(G_4)$, $L_5(v)=L_1(v)$ for $v\in V(G_1)\setminus\{v_2,v_4,x,y\}$,
    and~$L_5(v)=L_1(v)\cup \{7,8\}$ for $v\in \{v_2,v_4,x,y\}$.  Then~$G_5$ is
    $L$-colorable for every half-list assignment~$L$, but not
    $(L_5:2)$-colorable.
\end{lemma}
\begin{proof}
    Let~$L$ be a half-list assignment for~$G_5$. The gadget~$(G_4,L_5|V(G_4))$ is
    $(v_1,v_3,\{w_{1,3},w_{2,3}\})$-relaxed by Lemma~\ref{lemma-force78exact}.
    Suppose first that~(i) holds for the restriction of~$L$ to~$G_4$
    (with~$S=\{w_{1,3},w_{2,3}\}$), and let~$\psi_0$ be the corresponding
    $L$-coloring of~$\{v_1,v_3\}$.  Greedily extend~$\psi_0$ to an $L$-coloring
    $\psi$ of~$G_1$.  Choose~$c_1\in L(w_{1,3})\setminus
    \{\psi(v_2),\psi(v_4)\}$ and~$c_2\in L(w_{2,3})\setminus
    \{\psi(x),\psi(y)\}$.  By~(i) for~$G_4$, there exists an $L$-coloring
    of~$G_4$ that extends~$\psi_0$ and assigns to~$w_{i,3}$ the color~$c_i$ for
    each~$i\in\{1,2\}$.  This yields, together with~$\psi$, an $L$-coloring
    of~$G_5$.

Suppose next that~(ii) holds for the restriction of~$L$ to~$G_4$
    (with~$S=\{w_{1,3},w_{2,3}\}$), and let~$\psi_0$ be the corresponding
    $L$-coloring of~$\{w_{1,3},w_{2,3}\}$.  Note that $L(v_1)=L(v_3)$ in this
    case.  Corollary~\ref{cor-core} implies that~$G_1$ has an $L$-coloring
    $\psi$ such that $\psi(v_2)\in L(v_2)\setminus\{\psi_0(w_{1,3})\}$,
    $\psi(v_4)\in L(v_4)\setminus\{\psi_0(w_{1,3})\}$, $\psi(x)\in
    L(x)\setminus\{\psi_0(w_{2,3})\}$, and~$\psi(y)\in
    L(y)\setminus\{\psi_0(w_{2,3})\}$.  By~(ii), the restriction of~$\psi\cup
    \psi_0$ to~$\{v_1,v_3,w_{1,3}, w_{2,3}\}$ extends to an $L$-coloring
    of~$G_4$, which together with~$\psi$ gives an $L$-coloring of~$G_5$.

It remains to show that~$G_5$ is not $(L_5:2)$-colorable.  If~$\varphi$ were an
    $(L_5:2)$-coloring of~$G_5$, then by Lemma~\ref{lemma-force78exact} we
    would have $\varphi(w_{1,3})=\varphi(w_{2,3})=\{7,8\}$, and thus the
    restriction of~$\varphi$ to~$G_1$ would be an $(L_1:2)$-coloring, thereby
    contradicting Corollary~\ref{cor-core}.
\end{proof}

\paragraph{The final graph.} We are now in a position to prove Theorem~\ref{thm-main} by simply using a standard construction to ensure uniform lengths of lists.

\begin{proof}[Proof of Theorem~\ref{thm-main}]
Let~$G$ be a graph and~$L'$ an assignment of lists of size~$8$ obtained as follows.
Let~$K$ be a clique with vertices~$r_1, \dotsc, r_4$, and let~$L'(r_1)=\dotsb=L'(r_4)=\{9,\dotsc, 16\}$.
For every $(L':2)$-coloring $\psi$ of~$K$, let~$G_\psi$ be a copy of the graph~$G_5$ from the gadget $(G_5,L_5)$ of Lemma~\ref{lemma-main},
    and for each vertex~$v\in V(G_\psi)$ such that $|L_5(v)|=2k$ with~$k\in\{2,3\}$, we add the edges~$vr_1, \dotsc, vr_{4-k}$ and let
$L'(v)=L_5(v)\cup\bigcup_{i=1}^{4-k}\psi(r_i)$.   If~$G$ had an $(L':2)$-coloring~$\varphi$, then letting~$\psi$ be
the restriction of~$\varphi$ to~$K$, the restriction of~$\varphi$ to~$G_\psi$ would be an $(L_5:2)$-coloring of~$G_5$,
    thereby contradicting Lemma~\ref{lemma-main}.

    Consider now a list assignment~$L$ for~$G$ such that $|L(v)|= 4$ for every~$v\in V(G)$.
    Let~$\varphi$ be any $L$-coloring of~$K$.  For each
    $(L':2)$-coloring~$\psi$ of~$K$, let~$L_{\psi}$ be the list assignment
    for~$G_\psi$ obtained by, for each~$v\in V(G_\psi)$, removing the colors of
    neighbors in~$K$ according to~$\varphi$, and possibly removing further
    colors to ensure that $|L_{\psi}(v)|=|L_5(v)|/2$.  By Lemma~\ref{lemma-main}, the
    graph~$G_\psi$ has an $L_{\psi}$-coloring.  The union of these colorings
    and~$\varphi$ yields an $L$-coloring of~$G$.
\end{proof}

\paragraph{Concluding remarks.} It follows from Theorem~\ref{thm-main} that for
each integer~$a\ge4$, there exists a graph~$G_a$ that is~$a$-choosable but
not~$(2a:2)$-choosable---if we have such a graph~$G_a$, taking the disjoint
union of~$\binom{2(a+1)}{2}$ copies of~$G_a$ and adding a vertex adjacent to
all other vertices yields~$G_{a+1}$, by an argument analogous to the list
uniformization procedure used for the proof of Theorem~\ref{thm-main}. It is
natural to ask whether there exists a graph that is $3$-choosable but not
$(6:2)$-choosable.  We believe this to be the case; in particular,
Corollary~\ref{cor-core} only requires lists of size at most~$6$.  However, it
does not seem easy to construct a gadget that satisfies the properties stated
in Lemma~\ref{lemma-stagad} without using a vertex with a list of size~$8$.

%%% AUTHOR:
%%% Bibliography goes here. Note that the arXiv cannot process bibtex
%%% or biber bibliographies.  Example of acceptable bibliograpy format:
\bibliographystyle{amsplain}

%% AUTHOR: You can generate such a bibliography from a .bib file by 
%% running pdflatex/bibtex/pdflatex/pdflatex and then pasting the .bbl file
%% between \begin{thebibliography} and \end{bibliography}

%%% AUTHOR: Include a short description of each author following the
%%% structure below. Use the same short tags used previously.  
%%% Use \imageat{} and \imagedot{} instead of "@" and "." in
%%% email addresses-this replaces the symbols with graphics to avoid 
%%% e-mail address harvesting from the .pdf file
\begin{aicauthors}
\begin{authorinfo}[zd]
    Zden\v{e}k Dvo\v{r}{\'a}k\\
  Associate Professor\\
  Computer Science Institute (CSI) of Charles University\\
  Prague, Czech Republic\\
  rakdver\imageat{}iuuk\imagedot{}mff\imagedot{}cuni\imagedot{}cz\\
  \url{https://iuuk.mff.cuni.cz/~rakdver/}
\end{authorinfo}
\begin{authorinfo}[xh]
  Xiaolan Hu\\
  Assistant Professor\\
  School of Mathematics and Statistics $\&$ Hubei Key Laboratory of Mathematical Sciences\\
    Central China Normal University, Wuhan, PR China.\\
  xlhu\imageat{}mail\imagedot{}ccnu\imagedot{}edu\imagedot{}cn\\
  \url{http://maths.ccnu.edu.cn/info/1042/15221.htm}
\end{authorinfo}
\begin{authorinfo}[jss]
  Jean-S{\'e}bastien Sereni\\
  Directeur de recherche\\
  Centre National de la Recherche Scientifique\\
  CSTB (ICube), Strasbourg, France.\\
  sereni\imageat{}kam\imagedot{}mff\imagedot{}cuni\imagedot{}cz\\
  \url{http://lbgi.fr/~sereni/index.html}
\end{authorinfo}
\end{aicauthors}

\end{document}